\documentclass[12pt]{amsart}
\usepackage{enumerate}
\usepackage{amsmath}
\usepackage{amssymb}
\usepackage{amsfonts}
\usepackage{amsthm, upref}
\usepackage{graphicx}
\usepackage[usenames, dvipsnames]{color}

    \numberwithin{equation}{section}

\usepackage{bm}
         \bmdefine\alphab{\mathbf{\alpha}}
\bmdefine\betab{\mathbf{\beta}}
\bmdefine\sigmab{\mathbf{\sigma}}

\newcommand{\comment}[1]{}
\newcommand{\eq}{\begin{equation}}
\newcommand{\en}{\end{equation}}

\theoremstyle{plain}
\newtheorem{thm}{Theorem}
\newtheorem{lemma}[thm]{Lemma}
\newtheorem{prop}[thm]{Proposition}
\newtheorem{cor}[thm]{Corollary}

\theoremstyle{definition}

\newtheorem{prb}{Problem}
\newtheorem{rmk}{Remark}
\newtheorem{conj}{Conjecture}

\newcommand{\bp}{\mathbf{p}}
\newcommand{\bq}{\mathbf{q}}
\newcommand{\bX}{\mathbf{X}}
\newcommand{\bM}{\mathbf{M}}
\newcommand{\FF}{\mbox{${\mathcal F}$}} 
\newcommand{\Ex}{{\mathbb E}}
\renewcommand{\Pr}{\mathbb{P}}
\newcommand{\dist}{\mathrm{dist}}
\newcommand{\Geometric}{\mathrm{Geometric}}
\newcommand{\bzero}{\mathbf{0}}
\newcommand{\sfrac}[2]{{\textstyle\frac{#1}{#2}}}
\newcommand{\ed}{\ \stackrel{d}{=} \ }
\DeclareMathOperator{\var}{var }
\DeclareMathOperator{\rank}{rank }

\begin{document}

\title[Fluctuations of Martingales]{Fluctuations of Martingales and Winning Probabilities of Game Contestants}
\author{David Aldous}
\address{University of California, Berkeley, CA 94720-3860}
\email{aldous@stat.berkeley.edu} 
\thanks{Aldous's research supported by N.S.F Grant DMS-0704159.}
\author{Mykhaylo Shkolnikov}
\address{University of California, Berkeley, CA 94720-3860}
\email{mshkolni@gmail.com}
\date{\today}
\begin{abstract}
Within a contest there is some probability $M_i(t)$ that contestant $i$ will be the winner, given information available at time $t$, and $M_i(t)$ must be a martingale in $t$.  Assume continuous paths, to capture the idea that relevant information is acquired slowly.
Provided each contestant's initial winning probability is at most $b$, one can easily calculate, without needing further model specification, the
expectations of the random variables $N_b = $ number of contestants whose winning probability ever exceeds $b$, and $D_{ab} = $ total number of downcrossings of the martingales over an interval $[a,b]$. The distributions of $N_b$ and $D_{ab}$ do depend on further model details, and we study how concentrated or spread out the distributions can be. The extremal models for $N_b$ correspond to two contrasting intuitively natural methods for determining a winner: progressively shorten a list of remaining candidates, or sequentially examine candidates to be declared winner or eliminated. 
We give less precise bounds on the variability of $D_{ab}$. We formalize the setting of infinitely many contestants each with infinitesimally small chance of winning, in which the explicit results are more elegant. A canonical process in this setting is the  Wright-Fisher diffusion associated with an infinite population of initially distinct alleles; we show how this process fits our setting and raise the problem of finding the distributions of $N_b$ and $D_{ab}$ for this process.
\end{abstract}

\maketitle

\section{Introduction}

Given a probability distribution 
$\bp = (p_i, i \geq 1)$ consider a collection of processes 
$(M_i(t), 0 \le t < \infty, i \ge 1)$ adapted to a filtration $(\FF_t)$ 
and satisfying \\
(i) $M_i(0) = p_i, i \ge 1$; \\
(ii) for each $t > 0$ we have 
$0 \le M_i(t) \le 1 \ \forall i$ and $\sum_i M_i(t) = 1$; \\
(iii) for each $i \ge 1$, 
$(M_i(t), t \ge 0)$ is a continuous path martingale;\\
(iv) there exists a random time $T < \infty$ a.s. such that, for some random $I$, 
$M_I(T) = 1$ and $M_j(T) = 0 \ \forall j \ne I$.

\medskip\noindent
Call such a collection a {\em $\bp$-feasible} process, and call the $M_i(\cdot)$ its {\em component martingales}. 
To motivate this definition, consider  contestants in a  contest which will have one winner at some random future time.  Then the probability $M_i(t)$ that contestant $i$ will be the winner, given information known at time $t$, must be a martingale as $t$ increases.  In this scenario all the assumptions will hold automatically except for path-continuity, which expresses the idea that
 information becomes known slowly.  

In view of the fact that continuous-path martingales have long been a central concept in mathematical probability, it seems curious that this particular 
``contest" setting  has apparently not previously  been studied systematically.  
Moreover the topic is appealing at the expository level because it can be treated at any technical level.
In an accompanying non-technical article  for undergraduates \cite{me-monthly} we show data on probabilities (from the Intrade prediction market) 
for candidates for the 2012 Republican U.S. Presidential Nomination.  The data is observed values of the variables $N_b$ and $D_{ab}$ below, and 
one can examine the question of whether there was an unusually large number of candidates that year whose fortunes rose and fell substantially.
In this paper, the proof in section \ref{sec:Nb} of 
distributional bounds on $N_b$ is mostly accessible to a student taking a first course in 
continuous-time martingales, and  subsequent sections slowly become more technically sophisticated.

The starting point for this paper is the observation that there are certain random variables 
associated with a $\bp$-feasible process whose {\em expectations} do not depend on the actual joint distribution of the component martingales, and indeed depend very little on $\bp$. 
For $0<a<b<1$ consider
\[ N_b := \mbox{ number of $i$ such that } \sup_t M_i(t) \ge b \]
\[ D_{a,b}:= \mbox{ sum over $i$ of the number of downcrossings of 
$M_i(\cdot)$ over $[a,b]$.} \]
Straightforward uses of the optional sampling theorem 
(described verbally in \cite{me-monthly} as gambling strategies) establish
\begin{lemma}
\label{L1}
If $\max_i p_i \le b$ then for any $\bp$-feasible process,
\[ \Ex [N_b] = 1/b, \quad \Ex [D_{a,b}] = (1-b)/(b-a) . \]
\end{lemma}
In contrast, the {\em distributions} of $N_b$ and $D_{a,b}$ will depend on the joint distributions of the component martingales, and one goal of this paper is to  study the extremal possibilities.  
Here is our result for $N_b$.
\begin{prop}
\label{P1}
(a)  If $\max_i p_i \le b$ then there exists a $\bp$-feasible process for which 
the distribution of $N^\bp_b$ is supported on the integers 
$\lfloor 1/b \rfloor$ and $\lceil 1/b \rceil$ bracketing its mean $1/b$. \\
(b) There exists a family, that is a $\bp$-feasible process for each $\bp$, such that 
the distributions of $N^\bp_b$ satisfy
\begin{equation}
 \dist(N^\bp_b) \to \Geometric(b) \mbox{ as } \max_i p_i \to 0 . 
\label{NG}
\end{equation}
(c) Any possible  limit distribution for $N^\bp_b$ as 
$\max_i p_i \to 0$ has variance at most 
$(1-b)/b^2$, the variance of $\Geometric(b)$.
\end{prop}
Clearly the distribution in (a) is the ``most concentrated" possible, 
and part (c) gives a sense in which 
 the $\Geometric(b)$ distribution  is the ``most spread out" distribution possible. 
The proof will be given in section  \ref{sec:Nb}.  
The construction for (a) formalizes the idea that we maintain a list of 
candidates still under consideration, and at each stage choose one 
candidate to be eliminated. 
The construction for (b) formalizes the idea that we examine candidates sequentially, 
deciding  to declare the current candidate to be the winner or  to be eliminated.  
 Returning briefly to the theme that this topic is amenable to popular exposition, with some imagination one can relate these two alternate ideas to those used in season-long television shows.
Shows like {\em Survivor} overtly follow the idea for (a), whereas the idea for (b) would correspond to a variant of 
{\em \ldots \ldots Millionaire} in which contestants were required to
 try for the million dollar prize and where the season ends when the 
prize is won.

We give an analysis of downcrossings $D_{ab}$ in section \ref{sec:down}, though 
with less precise results.
The  construction that gave the Geometric limit distribution for $N_b$ in 
(\ref{NG}) also gives a Geometric limit distribution for $D_{ab}$ 
(Proposition \ref{LDab}).  We conjecture this is the maximum-variance possible limit, but can give only a weaker 
bound in Proposition \ref{prop_UBD}.  
As for minimum-variance constructions, Proposition \ref{LBDcrit} shows
one can construct feasible processes for which, in the limit as $b \to 0$ 
with $a/b$ bounded away from $1$,
the variance of $D_{ab}$ is bounded by a constant depending only on 
$a/b$. 
The case $a/b \approx 1$ remains mysterious, but prompts novel open problems 
about negative correlations for Brownian local times -- see section \ref{sec:OPext}.

\subsection{$\bzero$-feasible processes}
\label{sec:0-int}
As a second goal of this paper, it seems intuitively clear that the concept 
of $\bp$-feasible process can be taken to the limit as $\max_i p_i \to 0$, 
to represent the idea of starting with an infinite number of  contestants each with only infinitesimal chance of winning.
Informally, we define a $\bzero$-feasible process as a process with the properties:

(i) for each $t_0 > 0$, conditional on $M_i(t_0) = p_i, i \geq 1$, the process 
$(M_i(t_0+t),  0 \le t < \infty, i \ge 1)$ is a $\bp$-feasible process;

(ii) $\sup_i M_i(t) \to 0 \mbox{ a.s.  as } t \downarrow 0$.

\noindent
There is some subtlety in devising a precise definition, 
which we will give in section \ref{sec:zero}.  Once this is done 
we can deduce results for general $\bzero$-feasible processes as 
 limits of results for $\bp$-feasible processes under the regime 
 $\max_i p_i \to 0$, and also we can construct specific $\bzero$-feasible processes by splicing together specific $\bp$-feasible processes under the same regime (Proposition \ref{Prop:consistent}).

By eliminating any dependence on $\bp$,  results often 
become cleaner for $\bzero$-feasible processes. 
For instance Proposition \ref{P1} becomes
\begin{cor}
\label{C1}
(a)  There exists a $\bzero$-feasible process such that, for each $0<b<1$, 
the distribution  $N_b$ is supported on the integers 
$\lfloor 1/b \rfloor$ and $\lceil 1/b \rceil$ bracketing its mean $1/b$. \\
(b) Given $0<b_0<1$, there exists a  $\bzero$-feasible process  such that,  for each $b_0 \le b<1$, 
$N_b$ has $\Geometric(b)$ distribution. \\
(c) Moreover for any $\bzero$-feasible process and any $0<b<1$ the variance of $N_b$ 
is at most $(1-b)/b^2$, the variance of $\Geometric(b)$.
\end{cor}

Setting aside the ``extremal" questions we have discussed so far, 
another motivation for considering the class of $\bzero$-feasible processes
is that there is one particular such process which we regard intuitively 
as the ``canonical" choice, and this is the 
{\em $\bzero$-Wright-Fisher process} 
discussed  in section \ref{sec:0WF}.  
This connection between (a corner of) the large literature on processes inspired by population genetics and our game contest setting seems not to have been developed before.
In particular, questions about the fluctuation behavior of the 
$\bzero$-Wright-Fisher process -- the distributions of 
$N_b$ and $D_{ab}$ -- arise more naturally in the contest setting, 
though it seems hard to get quantitative estimates of these distributions.

\section{Preliminary observations}
\label{sec:back}
\subsection{The downcrossing formula}
In our setting of a continuous-path martingale $M(\cdot)$ ultimately stopped at $0$ or $1$, 
recall the ``fair game  formula"
\begin{equation}
 \Pr (M(t) \mbox{ hits $b$ before $a$ } \vert M(0) = x) = 
\sfrac{x-a}{b-a}, \ 0 \le a \le x \le b \le 1 
\label{fair-game}
\end{equation}
from which one can readily derive the well known formula for the expectation of the number $D$ of downcrossings of $M(\cdot)$ over $[a,b]$:
 for $0 \le a \le b \le 1$,
\begin{eqnarray}
 \Ex [D \vert M(0) = x] &=& \sfrac{x(1-b)}{b-a} \quad \mbox{ if } 0 \le x \le b \label{eq-3}\\
                                   &=& \sfrac{b(1-x)}{b-a} \quad \mbox{ if } b \le x \le 1 \label{eq-4}.
\end{eqnarray}
Moreover, starting from $b$ there is a modified Geometric distribution 
for $D$:
\begin{eqnarray}
\Pr(D = 0 | M(0) = b) &=& \sfrac{b-a}{1-a} \nonumber\\
\Pr(D = d | M(0) = b) &=& \sfrac{1-b}{1-a} 
\ \left(\sfrac{a(1-b)}{b(1-a)} \right)^{d-1} \ 
\left( 1 - \sfrac{a(1-b)}{b(1-a)} \right) , \quad d \ge 1 .\label{eq-5}
\end{eqnarray}

\subsection{The multivariate Wright-Fisher diffusion}
\label{sec:multi}
Textbooks introducing discrete time martingales often use as an example (e.g. \cite{lange-text} Example 10.2.6) the discrete-time 
Wright-Fisher model for genetic drift of a single allele. Note that throughout what follows, we consider only the case of no mutation and no selection.
It is classical that the infinite population limit of the $k$-allele model 
is the multivariate Wright-Fisher diffusion on the  $k-1$-dimensional simplex, that is with generator 
\begin{equation}
\sfrac{1}{2} \sum_{i,j = 1}^k  x_i(\delta_{ij} - x_j) \frac{\partial^2}{\partial x_i \partial x_j}. 
\label{WF-generator}
\end{equation} 
Each component is a martingale, the one-dimensional diffusion on 
$[0,1]$ with drift rate zero and variance rate $x(1-x)$. 
There has been extensive work since the 1970s on  the
infinitely-many-alleles case, but this has focussed on the case of positive 
mutation rates to novel alleles, in which case the martingale property no longer holds.  In our setting 
(no mutation and no selection) 
it is straightforward to show directly (see section \ref{sec:0WF})
that for any 
$\bp = (p_i, i \geq 1)$ with countable support there exists what we 
will call the $\bp$-Wright-Fisher process, 
the infinite-dimensional diffusion with generator analogous to 
(\ref{WF-generator}) starting from state $\bp$, and that this is a 
$\bp$-feasible process.  
So we know that $\bp$-feasible processes do actually exist,  
and these $\bp$-Wright-Fisher processes will be useful ingredients in later constructions. 
(When $\bp$ has finite support we could use instead Brownian motion on the finite-dimensional simplex, 
whose components are killed at $0$ and $1$, but this does not extend so readily to the 
infinite-dimensional setting).   

It is convenient to adopt from genetics the phrase {\em fixation time} for the time $T$ at which the winner is determined.

\subsection{Constructions using Wright-Fisher}
In a Wright-Fisher diffusion we have 
$\sum_i M_i(t) \equiv 1$, but trivially we can consider a rescaled 
Wright-Fisher diffusion for which $\sum_i M_i(t)$ is a prescribed constant.

Our constructions of feasible processes typically proceed in stages.  
Within a stage we may declare that
some component martingales are ``frozen" 
(held constant) and the others evolve as a rescaled Wright-Fisher process.
In particular if only two component martingales are unfrozen, say at the start $S$ of the stage we have
$M_i(S) = x_i$ and $M_j(S) = x_j$,
 then during the stage we have a
 ``reflection coupling" with $M_i(t) + M_j(t) = x_i +x_j$, and we can choose to continue the stage until the processes reach $x_i+x_j$ and $0$, 
or we can choose to stop earlier. 

An alternative construction method is to select one component martingale $M_i(S)$ at the start of the stage, let $(M_i(\cdot), 1 - M_i(\cdot))$ evolve as the two-allele Wright-Fisher process during the stage, and set $M_j(\cdot) = \frac{M_j(S)}{1-M_i(S)} \times 
(1 - M_i(\cdot))$.  We describe this construction by saying that 
the processes $(M_j(\cdot), j \ne i)$ are {\em tied}.  

Both constructions clearly give continuous-path martingale components.

The results in sections \ref{sec:Nb} and \ref{sec:down} are based on  
concrete calculations and constructions, though  in applying them to $\bzero$-feasible processes
 we ``look ahead" and quote results from later 
(Propositions \ref{Prop:consistent} and \ref{P-embed}) which are designed for this purpose, formalizing the intuitive description from section 
\ref{sec:0-int} so as to allow results to be easily interchanged between $\bp$-feasible and 
$\bzero$-feasible processes.

\section{Proofs of distributional bounds on $N_b$}
\label{sec:Nb}
\begin{proof}[Proof of Proposition \ref{P1}(a)]  
Fix $b$. 
Run a Wright-Fisher process started at $\bp$ until some $M_i(\cdot)$ reaches $b$.  
Freeze that $i$ and run the remaining processes as  rescaled 
Wright-Fisher until some other $M_j(\cdot)$ reaches $b$. 
Freeze that $j$ and continue. 
After a finite number of such stages we must reach a state where all component martingales except one are frozen at $b$ or at $0$, and the remaining one is in $[0,b]$.  Because $\sum_i M_i(t) \equiv 1$  the number frozen at $b$ must be $\lfloor 1/b \rfloor$ and the 
remaining martingale must be at $1 - b \lfloor 1/b \rfloor$. 
Finally, unfreeze and  run from  this configuration to fixation as 
Wright-Fisher.  Clearly $N_b$ takes only the values $\lfloor 1/b \rfloor$ 
and $\lceil 1/b \rceil$. 
\end{proof}

\begin{proof}[Proof of Corollary \ref{C1}(a)] 
This construction is similar to that above, but is closer to our earlier informal description
``maintain a list of 
candidates still under consideration, and at each stage choose one 
candidate to be eliminated".

For each integer $m \ge 2$,
we will define a stage  which starts with $m$ 
component martingales at $1/m$, and ends with $m-1$ of these martingales at $1/(m-1)$ 
and the other frozen at $0$.
To construct this stage,
run as Wright-Fisher until some $M_i(\cdot)$ reaches $1/(m-1)$.  
Freeze that $i$ and run the remaining martingales as  rescaled 
Wright-Fisher until some other $M_j(\cdot)$ reaches $1/(m-1)$. 
Freeze that $j$ and continue. 
Eventually we must reach a state where  $m-1$ martingales  are frozen at $1/(m-1)$  and the remaining process is $0$. 
This stage takes some random time $\tau_m$ with finite expectation; 
without needing to calculate it, we can simply rescale time so that 
$\Ex [ \tau_m] = 2^{-m}$.  

Intuitively, we simply put these stages together, to obtain a 
$\bzero$-feasible process in which, for each $M \ge 1$, at time 
$\sum_{m>M} \tau_m$ there are exactly $M$ martingales at $1/M$. 
Proposition \ref{Prop:consistent} formalizes this construction.
This process satisfies the assertion of the Corollary for each $b = 1/M$, 
and then for general $b$  because $N_b$ is monotone in $b$.
\end{proof}

\begin{rmk}
\label{R:survivor}
Let us call a $\bzero$-feasible process with the property above, 
that for each $M \ge 2$ there is a time at which there are exactly $M$ 
component martingales each at value $1/M$, a {\em Survivor} process.
We placed the proof here to illustrate the technical issue  arising in making precise the construction of such a $\bzero$-feasible process, which 
is to arrange consistent labeling of each component martingale across the different stages.  
The point is that the one out of the $M$ that does not reach $1/(M-1)$ is a uniform random pick, so
we cannot just label them as $1,\ldots,M$ for each $M$. 
\end{rmk}

\begin{proof}[Proof of Proposition \ref{P1}(b)]
  Fix $b$.
Write $\bp$ in ranked order $p_1 \ge p_2 \ge \ldots$, 
and write $J$ for the first term (if any) such that 
$p_{J+1}/(1 - p_1 - \ldots - p_J) > b$.  

We use the ``tied" construction from the start of this section.
Run $(M_1(\cdot), 1-M_1(\cdot))$ as Wright-Fisher started from $(p_1,1-p_1)$ and stopped 
at $S_1 := \min \{t: M_1(t) = 0 \mbox{ or } 1\}$, and for $i \ge 2$ set
\[ M_i(t) = \sfrac{p_i}{1-p_1} (1 - M_1(t)), \ 0 \le t \le S_1 . \] 
So $M_i(\cdot)$ is a martingale on this time interval.
Note that if $J \ne 1$ then no $M_i(\cdot)$ can reach $b$ before time 
$S_1$, for $i \ge 2$. 

If $M_1(S_1) = 1$ the process stops.  
If $M_1(S_1) = 0$ then for $i \ge 2$ we have $M_i(S_1) = p_i/(1-p_1)$.  
For $t \ge S_1$ 
run $(M_2(\cdot),1-M_2(\cdot))$ as Wright-Fisher started from 
$(\frac{p_2}{1-p_1}, \frac{1-p_1-p_2}{1-p_1})$ and stopped 
at $S_2 := \min \{t: M_2(t) = 0 \mbox{ or } 1\}$, and for $i \ge 3$ set
\[ M_i(t) = \sfrac{p_i}{1-p_1- p_2} (1 - M_2(t)), \ S_1 \le t \le S_2 . \]
If $J \ne 2$ then no $M_i(\cdot)$ can reach $b$ before time 
$S_2$, for $i \ge 3$. 

Continue in this way to define processes 
$(M_j(t), \ S_{j-1} \le t \le S_j)$ for $1 \le j \le J$, or until some $M_j(\cdot)$ reaches $1$ and the whole process stops. 
If the process has not stopped by time $S_J$, continue in an arbitrary manner, which makes the resulting process $\mathbf{p}$-feasible. Note that, if $M_j(\cdot)$ reaches $b$, then with probability exactly $1/b$ it will reach $1$, and that with probability $1 - \sum_{j \le J} p_j$ the process has not stopped by time $S_J$.  

Write $N_b^{(J)} = $ 
number of martingales $j \le J$ that reach $b$.
 We can now apply Lemma \ref{Lb} below to $Z = N_b^{(J)}$, and deduce that 
$N_b^{(J)} \le Z^\prime \ed \mathrm{Geometric}(b)$ with $Z^\prime$ constructed in Lemma \ref{Lb}.  
Then
\begin{eqnarray*}
\Pr (N_b \ne Z^\prime) 
&\le& \Ex [ |N_b - Z^\prime| ]\\
&\leq & \Ex [ Z^\prime - N_b^{(J)} ] + \Ex [N_b - N_b^{(J)}] \\
&=& b^{-1}(1 - \sum_{j \le J} p_j) + b^{-1}(1 - \sum_{j \le J} p_j) \\
&=& 2 b^{-1}\sum_{j>J} p_j .
\end{eqnarray*}
Finally, as $\bp$ varies we have
\[ \mbox{ if } \max_i p_i \to 0 \mbox{ then } \sum_{j > J(\bp)} p_j \to 0\]
establishing the limit result (\ref{NG}).
\end{proof}

\begin{lemma}
\label{Lb}
Given $0<b<1$ and probabilities $q_i, 1 \le i \le J$ define a
counting process by: for each $i$, given not yet terminated,

with probability $q_i b$, increment count by $1$ and terminate;

with probability $q_i (1-b)$, increment count by $1$ and continue;

with probability $1 - q_i$, continue.\\
Let $Z$ be the value of the counting process after step $J$ or at time $T$ (the termination time, if any), whichever occurs first. 
Then there exists $Z^\prime \ed \mathrm{Geometric}(b)$ such that 
$Z \le Z^\prime$.
\end{lemma}
\begin{proof} 
Augment the process by setting $q_i = 1, \ i > J$ and follow the algorithm for all $i\geq1$. The process must now terminate at some a.s. finite time $T^\prime$, at which time the value $Z^\prime$ of the counting process has exactly 
Geometric($b$) distribution.   
\end{proof}

\begin{proof}[Proof of Proposition \ref{P1}(c)]
Fix $b$ and for $k \ge 1$ let $S_k \le \infty$ be the first time at which 
$k$ distinct component martingales have reached $b$.  If  $N_b \ge k$, 
then at time $S_k$ one martingale takes value $b$, the other $k-1$ that previously reached $b$ take some values $Z_1,\ldots, Z_{k-1}$, and the remaining martingales take some 
values $M_j(S_k)< b$.  The chance that such a remaining martingale subsequently reaches $b$ equals $M_j(S_k)/b$, and so,  
on $\{N_b \ge k\}$, 
\begin{equation}
 \Ex [N_b - k \vert \FF_{S_k}] = b^{-1} \sum_j M_j(S_k) 
= b^{-1} \left(1 - b - \sum_{j=1}^{k-1} Z_j \right) 
\le \frac{1-b}{b} . 
 \label{NZ}
\end{equation}
So 
\[ \Ex [  (N_b - k)_+] \le  \sfrac{1-b}{b} \  \Pr (N_b \ge k) \]
and summing over $k \ge 1$ gives
\[ \Ex \left[ \sfrac{N_b(N_b - 1)}{2}  \right] \le \sfrac{1-b}{b} \Ex [N_b] = \sfrac{1-b}{b^2} .\]
Finally, 
\[ \var ( N_b) = 2 \Ex [ \sfrac{N_b(N_b - 1)}{2} ] + \Ex [N_b] - (\Ex N_b)^2 \le \sfrac{1-b}{b^2} .\]
\end{proof}

For later use (section \ref{sec:zero}) note that to have equality in the final display above we need equality in (\ref{NZ}), implying each $Z_j = 0$, that is the martingale components that previously reached $b$ have all reached zero.  We deduce 
\begin{cor}
\label{C-b}
If, for a $\bp$-feasible process, $N_b$ has Geometric($b$) distribution, then there is no time at which more than one component martingale is in $[b,1]$.
\end{cor}

\begin{proof}[Proof of Corollary \ref{C1}(c)] 
This follows from Proposition \ref{P1}(c) and the definition (section \ref{sec:zero}) of $\bzero$-feasible process via embedded 
$\bp$-feasible processes. 
\end{proof}

\begin{proof}[Proof of Corollary \ref{C1}(b)]
Given $b_0$, consider the vector $\bp$ of Geometric probabilities with 
\begin{equation}
p_i = b_0(1-b_0)^{i-1}, i \ge 1.
\label{b0}
\end{equation}
The construction in the proof of Proposition \ref{P1}(b) and its analysis show that for this 
$\bp$-feasible process and any $b \ge b_0$ we have $J = \infty$ and that $N_b$ has $\mathrm{Geometric}(b)$ distribution.  
So it is enough to show that there exists a $\bzero$-feasible process and a stopping time at which the values of the component martingales are 
$\bp$.  But Proposition \ref{P-embed} shows this is true for every $\bp$.
\end{proof}

\section{Distributional bounds on downcrossings}
\label{sec:down}
\subsection{The large spread setting}
\label{sec:large_spread}
\begin{prop}
\label{LDab}
Given $b_0>0$, there exists a  $\bzero$-feasible process  such that 
$D_{ab} + 1$ has $\Geometric(\frac{b-a}{1-a})$ distribution, for each $b_0 \le b < 1$ and $0<a<b$.  
\end{prop}
The corresponding result (cf. Proposition \ref{P1}(b)) holds for 
 $\bp$-feasible processes in the limit as $\max_i p_i \to 0$.
\begin{proof} 
 As in the proof of Corollary \ref{C1}(b), we may start with 
the Geometric($b_0$) distribution $\bp$ at (\ref{b0}) and use 
the construction in the proof of Proposition \ref{P1}(b). 
Every time a martingale component reaches $b$, the other components must be at positions 
\[ (1-b) \  b_0(1-b_0)^{i-1}, i \ge 1. \]
Similarly, each time the component completes a downcrossing of $[a,b]$  the other components must be at positions  
\[ (1-a) \  b_0(1-b_0)^{i-1}, i \ge 1. \] 
The event that there are no further downcrossings is the event that, after the next time some component reaches $b$, it then reaches $1$ before $a$, and this has probability $(b-a)/(1-a)$ by (\ref{fair-game}).  So
\[ \Pr (D_{ab} = i \vert D_{ab} \ge i) = (b-a)/(1-a), \ i \ge 1. \]
By the same argument $\Pr(D_{ab} = 0 ) = (b-a)/(1-a)$.
\end{proof}

\medskip 
The variance of the $\Geometric(\frac{b-a}{1-a})$ distribution can be written as
\begin{equation}
 \left( \sfrac{1-b}{b-a} \right)^2 + \sfrac{1-b}{b-a}  . \label{Gab-var}
\end{equation}
It is natural to guess, analogous to Corollary \ref{C1}(c), that this is an upper bound on the variance of $D_{ab}$ in any $\bzero$-feasible process.
\begin{conj}
\label{conj_1}
For any $\bzero$-feasible process, 
\[\var (D_{ab}) \le  \left( \sfrac{1-b}{b-a} \right)^2 + \sfrac{1-b}{b-a}  . \]
\end{conj}
The following result establishes a weaker bound.  
One can check that in the $a \uparrow b$ limit this bound is 
first order asymptotic to 
$ ( \sfrac{1-b}{b-a} )^2$, which coincides with the 
first order asymptotics in (\ref{Gab-var}).
\begin{prop}
\label{prop_UBD}
For any $\mathbf{0}$-feasible process and any $0<a<b<1$, 
\[
\var (D_{ab})\leq \left(\left(\frac{1-b}{b-a}+2\,\frac{(1-b)^2}{(b-a)^2}+\mu\right)^{1/2}+\mu^{1/2}\right)^2
-\frac{(1-b)^2}{(b-a)^2}
\]
where $\mu :=\min((2-b)/b^2,1/a^2)$.
\end{prop}  
\begin{proof}
Fix $0<a<b<1$ and consider an arbitrary $\bzero$-feasible process.
Call a particular component martingale at a particular time {\em active} if it is potentially 
part of a downcrossing of $[a,b]$. 
That is, the martingale is initially inactive; it becomes active if and when it first reaches $b$; it becomes inactive if and when it next reaches $a$; 
and so on.  So a martingale at $x$ is always active if $x>b$, is always 
inactive if $x<a$, but may be active or inactive if $a<x<b$.

Given that a particular martingale is currently at $x$, the mean number of future downcrossing completions equals, by (\ref{eq-3}, \ref{eq-4}) 
\[ \sfrac{x (1-b)}{b-a} \mbox{ if inactive; }
\quad 
\quad 
\sfrac{(1-x) b}{b-a} \mbox{ if active.}
\]
Analogously to the proof of Proposition \ref{P1}(c), consider the time $S_k$ at which the $k$'th downcrossing has been completed. 
On $\{S_k < \infty\}$,
\[ (b-a) \Ex [D_{ab} - k|\FF_{S_k}] = 
(1-b) \sum_{i \ inactive} M_i(S_k) 
+ b \sum_{i \ active} (1 - M_i(S_k)) \]
and because $\sum_i M_i(\cdot) = 1$ this becomes
\[ b + (b-a) \Ex [D_{ab} - k|\FF_{S_k}] = 
\sum_{i \  inactive}  M_i(S_k) + \sum_{i \ active} b . \]
The number of active martingales at time $S_k$ is at 
most $N_b^{(k)}:=$ number of martingales that reached $b$ before time $S_k$. 
So the right side cannot be larger than the value taken 
when $\min(N_b^{(k)},1/a)$ active martingales take values just above $a$ and the remaining value of $1-a\min(N_b^{(k)},1/a)$ is distributed among the inactive martingales. This gives the upper bound
\[ b + (b-a) \Ex [D_{ab} - k|\FF_{S_k}] 
\leq 1-a\min(N_b^{(k)},1/a)+b\min(N_b^{(k)},1/a) 
\mbox{ on } \{S_k < \infty\} . \]
The event $\{S_k < \infty\}$ is the event 
$\{D_{ab} \ge k \}$, so taking expectations and rearranging gives
\[ \Ex[(D_{ab}-k)_+] \le
 \sfrac{1-b}{b-a}\,\Pr (D_{ab}\geq k)+\Ex[\min(N_b^{(k)},1/a)\,\mathbf{1}_{\{D_{ab}\geq k\}}] 
. \]
Because $N_b^{(k)} \le N_b$, summing over all $k \ge 1$ gives
\begin{equation}
\label{basic_ineq}
\sfrac{1}{2}\Ex[D_{ab}(D_{ab}-1)]\leq \sfrac{1-b}{b-a}\,\Ex[D_{ab}]+\Ex[\min(N_b,1/a)\,D_{ab}]. 
\end{equation}
Apply the Cauchy-Schwarz inequality to the second summand on the right   side and use $\Ex[D_{ab}]=\frac{1-b}{b-a}$ to conclude
\begin{equation}
\Ex[D_{ab}^2]\leq\frac{1-b}{b-a}+2\,\frac{(1-b)^2}{(b-a)^2}+2\,\Ex\big[\min(N_b,1/a)^2\big]^{1/2}\,\Ex\big[D_{ab}^2\big]^{1/2}. 
\end{equation}
Next, for  positive constants $C_1,C_2$ we have the elementary implication
\[ \mbox{ if } 0 \le a\leq C_1+2\,C_2\sqrt{a} 
\mbox{ then } \sqrt{a} \leq \sqrt{C_1+C_2^2}+C_2 .  
\]
In our situation, this gives
\[
\sqrt{\Ex [D_{ab}^2] } \leq \left(\frac{1-b}{b-a}+2\,\frac{(1-b)^2}{(b-a)^2}+\Ex\big[\min(N_b,1/a)^2\big]\right)^{1/2}
+\Ex\big[\min(N_b,1/a)^2\big]^{1/2} .
\]
Using first Jensen's inequality and then the result 
(Corollary \ref{C1}(c)) that $\var (N_b) \le (1-b)/b^2$, we see
\[
\Ex [\min(N_b,1/a)^2]\leq\min(\Ex[N_b^2],1/a^2) 
\leq \min((2-b)/b^2, 1/a^2) 
\]
from which the inequality in the proposition readily follows. 
\end{proof}

\subsection{The small spread setting}
\label{sec:small_spread}
Proposition \ref{P1}(a) showed that the spread of $N_b$ could be very small.
To see that the case of $D_{ab}$ must be somewhat different, recall that for a martingale component which reaches $b$, its number of downcrossings 
has the modified Geometric distribution (\ref{eq-5}) with mean 
$b(1-b)/(b-a)$.  So if we fix $b$ and consider limits in distribution 
as $a \uparrow b$, we must obtain a limit of the form
\[ \frac{b-a}{b(1-b)} D_{ab} \  \to_d\ \  \sum_{i=1}^{N_b} \xi_i \]
where each $\xi_i$ has Exponential($1$) distribution.  
And although there will be some complicated dependence between 
$(N_b, \xi_1, \xi_2,\ldots)$ it is clear that the limit cannot be a constant, 
and therefore in any $\bp$-feasible process the variance of $D_{ab}$ 
as $a \uparrow b$
must grow at least as order $(b-a)^{-2}$.
We will not consider that case further here 
(but see an open problem in section \ref{sec:OPext}), instead turning to the case where 
$a/b$ is bounded away from $1$.  
Here, in a $\bzero$-feasible process,  $\Ex [D_{ab}]$ grows as order $1/b$ as $b \downarrow 0$. 
The next result shows there exist $\bzero$-feasible processes for which the variance of $D_{ab}$ remains $O(1)$.

The idea behind the construction is to exploit reflection coupling. 
For instance, starting with $2m$ components at $b$, a reflection coupling 
moves the process to a configuration with $m$ components at $a$ and $m$ at $2b-a$ while adding
 $m$ downcrossings; one can extend this kind of construction to make the process pass through a deterministic sequence of configurations while adding a deterministic number of downcrossings.
\begin{prop}
\label{LBDcrit}
For each $0 \le \alpha < 1$ there exists a constant  $C(\alpha)<\infty$ 
such that: given $0 < a_k < b_k \to 0$ with $a_k/b_k \to \alpha$, 
there exist $\mathbf{0}$-feasible processes  such that
\begin{equation}
\label{varLBD}
\limsup_k \ \mathrm{var}(D_{ak,b_k})\leq C(\alpha).
\end{equation}
\end{prop}  
\begin{proof} 
Fix $k$, set $(a,b)=(a_k,b_k)$ and with an abuse of notation write $\alpha=a_k/b_k$. 
By Proposition \ref{P-embed} we may assume we have a
$\bp_0$-feasible process, where $\bp_0$ has finite support 
and its components are in $(0,b)$.  

The proof makes repeated use of the following kind of construction.
Specify an interval $[a_0,b_0]$, freeze martingale components initially outside that interval, run the other components as a rescaled Wright-Fisher 
process and freeze them upon reaching $a_0$ or $b_0$ 
(typically there will be one component ending within $(a_0,b_0)$). 
Note this construction has a particular ``deterministic" property,
 that in the final random 
configuration $(M_i(t), i \ge 1)$ the ranked (decreasing ordered) values 
$\rank (M_i(t), i \ge 1)$ are non-random, determined by  the (ranked) 
initial values. 
This holds because $\sum_i M_i(t) = 1$.

The central idea of the proof
 is the following lemma.
\begin{lemma}
\label{L:core}
Write $K=K(\alpha)=6
\lfloor \sfrac{1}{1-\alpha} \rfloor  -1$.  There exists a $\bp_0$-feasible process  which reaches a configuration $\bp_1$ with at most 
%$(K+1)$ martingales above $0$, 
one martingale with value in $(b,1]$ and at most $K$ martingales taking values in $(0,b]$, having accomplished a deterministic number of downcrossings before that time.
\end{lemma}
\begin{proof}
We construct the process in stages.  
At the start of each stage, we consider  the first case in the list below 
which holds, and do the construction specified below for that case.
If no case holds then stop; note the property 
``at most $K$ martingales taking values in $(0,b]$" will then be satisfied.

{\em Case 1.} There are at least $1+\lfloor \sfrac{1}{1-\alpha} \rfloor $ martingales at $b$;

{\em Case 2.} There are at least $2 \lfloor \sfrac{1}{1-\alpha} \rfloor + 1$ active martingales in $(a,b)$;

{\em Case 3.}  There are at least $2 \lfloor \sfrac{1}{1-\alpha} \rfloor  + 1$ inactive martingales in $(a,b)$; 

{\em Case 4.} There are at least $\lfloor \sfrac{1}{1-\alpha} \rfloor $ martingales in $(0,a]$.

{\em Construction in case 1.}  We let the martingales at $b$ evolve according to the appropriately rescaled Wright-Fisher diffusion, while freezing all other martingales, and then freeze the evolving martingales that reach level $a$.  At least $\lfloor \sfrac{1}{1-\alpha} \rfloor $ martingales will reach level $a$, and exactly one will be above $b$.  Once all martingales are frozen, we let those at $a$ evolve as the rescaled Wright-Fisher diffusion until they reach $0$ or $b$.  Finally, if initially there were martingales above $b$, then we let all the martingales above $b$ evolve as the appropriate Wright-Fisher diffusion and freeze those that reach $b$. This procedure adds a deterministic number of downcrossings (all in the first step), and leaves exactly one martingale above $b$.
% and the number of martingales with values in $(0,b]$ becomes strictly %smaller than it was before.

{\em Construction in cases 2 and 3.}  In case 2 we let the active martingales in $(a,b)$ evolve until they either reach $a$ or $b$ and freeze them at that time. All except one of these martingales reach $a$ or $b$, so  either at least $\lfloor \sfrac{1}{1-\alpha} \rfloor + 1$ martingales end at $b$, or at least $\lfloor \sfrac{1}{1-\alpha} \rfloor $ martingales end at $a$, adding a deterministic number of downcrossings.  So the ending configuration will fit case 1 or case 4. 
In case 3 we do the same but with the inactive martingales instead; 
 no additional downcrossings are added.

{\em Construction in case 4.}  We let the martingales in $(0,a]$ evolve until they reach $0$ or $b$ and freeze them at that time. At least one of them must  reach $0$, and no additional downcrossings are added.  

\medskip

The sequence of stages must end, because:  in each case 4 stage at least
one martingale is stopped at $0$, and each case 1 stage creates at least one downcrossing, so there can be only a finite number of such stages; and each case 2 or 3 stage is followed by such a stage.

Moreover each stage is ``deterministic", in the previous sense that the ranked configuration at the end of the stage is determined by the ranked configuration at the start, and therefore the ranked configuration $\bp_1$ at the termination of the entire construction is non-random,  determined by the initial 
configuration $\bp_0$. 
This implies the total number of downcrossings is deterministic, because the number within each stage is determined by that stage's starting configuration.
As already mentioned, $\bp_1$ has the property 
``at most $K$ martingales taking values in $(0,b]$" by the termination condition.  The number of martingale components taking values in $(b,1]$ is at most $1$, because each case $1$ stage ends that way and the other cases do not allow components to exceed level $b$.
\end{proof}

In view of Lemma \ref{L:core}, to complete the proof of the proposition it suffices to show (\ref{varLBD}) for some $\mathbf{p}_1$-feasible process with $\mathbf{p}_1$ as in Lemma \ref{L:core}.  
In fact we can take an arbitrary such process.  
The point is that (as noted earlier) the number of downcrossings $D_{ab}$ has a representation 
of the form 
\[ D_{ab} = \sum_{i=1}^{N^*} G_i \]
where $N^*$ is the number of martingale components that hit $b$, and each $G_i$ has the modified Geometric distribution (\ref{eq-5}).
Without any knowledge of the dependence between 
$(N^*, G_1, G_2,\ldots)$, 
the fact $N^* \le K+1$ implies  
\[ \var (D_{ab} ) \le \Ex [D_{ab}^2] \le (K+1)^2 \Ex [G_1^2] .\]
It is easy to check that $\Ex [G_1^2]$ is bounded in the limit as $b \to 0$ with $a/b \to \alpha < 1$, 
and (\ref{varLBD}) follows.
\end{proof}

 \section{$\bzero$-feasible processes}
\label{sec:zero}
In section \ref{sec:def_0} we will give one formalization of the notion of a $\bzero$-feasible process
introduced informally in
section \ref{sec:0-int}, and in sections \ref{sec:construct} and 
\ref{sec:all_embed} we give results allowing one to relate constructions and properties of 
$\bzero$-feasible processes to those of $\bp$-feasible processes.

There are several possible choices for the level of generality 
we might adopt. 
The ``canonical" example of the $\bzero$-Wright-Fisher process, 
and the ``{\em Survivor}" process featuring in Corollary \ref{C1}(a),
 have the 
property that at times $t>0$ the  process has only finitely many 
non-zero components, so we could make this a requirement. 
Instead we will allow a countable number of non-zero components --  ``because we can". 
In the other direction, consider the construction of reflecting Brownian motion
$R(t)$ from standard Brownian motion $W(t)$ as 
\[ R(t) := W(t) - \min_{s \le t} W(s)  \]
and run the process until $R(\cdot)$ hits $1$.  
Within our setting, interpret this as saying that at time $t$ there is one contestant with chance $R(t)$ of winning, the remaining chance $1 - R(t)$ being split amongst an infinite number of  unidentified contestants each with only infinitesimal chance of winning.  Informally this is a
$\bzero$-feasible process such that 
\begin{equation}
\mbox{ $N_b$  has Geometric($b$) 
distribution for every $0<b<1$},
\label{G_all}
\end{equation}
 strengthening the assertion of 
Corollary \ref{C1}(b), 
but it does not fit our set-up which will require the unit mass to be 
split as a random discrete distribution at times $t>0$.  
In fact  Corollary \ref{C-b} implies that, within our formalization, no 
 $\bzero$-feasible process can have property (\ref{G_all}).
One could choose a more general set-up which allows such ``dust", as in the literature \cite{MR2253162} cited below, but we are not doing so.

The existing classes of processes in the literature with somewhat similar qualitative behavior
--  in the theory of  stochastic fragmentation and coagulation processes \cite{MR2253162} which studies partitions of unit mass into clusters, 
or in  population genetics inspired processes associated with 
Kingman's coalescent, are (to our knowledge) explicitly Markovian, 
in which context the question becomes determining the 
{\em entrance boundary} of a specific Markov process 
\cite{MR1808372,MR1112408}.  Our setting differs in that
 we wish to continue making only the ``martingale"  assumptions 
(ii,iii,iv) at the start of the Introduction, and we are seeking to define a class 
of processes.

The following observation shows that the most naive formalization 
does not work. 
\begin{lemma}
\label{L:naive}
Let $I$ be countable, 
There does not exist any process 
$(M_i(t), 0 \le t < \infty, i \in I)$ adapted to a filtration $(\FF_t)$ 
and satisfying \\
(i) for each $t > 0$ we have 
$0 \le M_i(t) \le 1 \ \forall i$ and $\sum_i M_i(t) = 1$; \\
(ii) for each $i $, 
$(M_i(t), t \ge 0)$ is a martingale;\\
(iii) $\sup_i M_i(t) \to 0$ a.s. as $t \downarrow 0$.
\end{lemma}
\begin{proof}
The martingale property implies $\Ex [M_i(t)]$ is constant in $t$.
But by (i) and (iii) we have 
\[ \lim_{t \downarrow 0} \Ex [M_i(t)] = 0  .\] 
So $\Ex[ M_i(t)]  = 0$ for all $i$ and $t$,  contradicting (i).
\end{proof}

\subsection{A formalization}
\label{sec:def_0}
The issue, indicated by  Lemma \ref{L:naive} above and the particular {\em Survivor} example in Remark \ref{R:survivor}, 
is to find a formalization which preserves the identity of martingale components
as $t$ varies.
The often used device of simply ranking (decreasing-ordering) components at each time $t$ does not work.
Our formalization combines ranking and a point process representation.  
This is admittedly somewhat {\em ad hoc}; a different but equivalent formalization 
is mentioned in Remark \ref{R1}.

A probability distribution $\bp$ with 
$p_1 \ge p_2 \ge p_3 \ge \ldots$ is called {\em ranked};
write $\nabla$ for the space of ranked 
probability distributions. For a general discrete distribution
$\bq = (q_j, j \in J)$ write
$\rank( \bq)$ for its decreasing ordering, where zero entries are omitted. 
More generally, 
for a collection $(A_j, j \in J)$ of objects with the same index set as 
$ (q_j, j \in J)$, write
$\rank(A_j, j \in J || \bq)$ for the collection re-ordered so that $\bq$ is ranked
(this is not completely specified if the values $q_j$ are not distinct, 
but the arbitrariness does not matter for our purposes).

Write $C_0$ for the space of continuous functions 
$f:[0,\infty) \to [0,1]$ with $f(0) = 0$.  
Consider a random point process on $C_0$.  
That is, a realization of the process is (informally) an unordered countable set $\{f_\alpha(\cdot)\}$ 
of functions or (formally) the counting measure associated with that set. 
We will use the  former notation, which is more intuitive.
We define a $\bzero$-feasible process to be a random point process 
$\{M_\alpha(\cdot)\}$ on $C_0$
such that 
\[ 0 \le M_\alpha(t) \le 1; \quad \sum_\alpha M_\alpha(t) = 1, \quad 
0<t<\infty\]
\[ \max_\alpha M_\alpha(t) \to 0 \mbox{ a.s. as } t \downarrow 0 \]
and with the following property. 
For each $t_0 > 0$ and each ranked $\bp$,  
\[ \mbox{ 
Conditional on $\rank(M_\alpha(t_0)) = \bp$ and on $\FF(t_0)$, the ranked process 
} \]
\begin{equation}
\rank(M_\alpha(t_0 + \cdot)) || \{M_\alpha(t_0)\}) \mbox{ is $\bp$-feasible.}
\label{rank-M}
 \end{equation}
In words,  given $t_0$ we simply label component martingales as $1,2,3,\ldots$ in decreasing order of their values at $t_0$, and we can use this labeling over $t_0 \le t < \infty$ to define a process 
$(M_i(t_0+t), t \ge 0, i \ge 1)$ 
which we require to be a $\bp$-feasible process, where $\bp$ is the ranked 
ordering of $(M_\alpha(t_0))$.  
For $\FF_t$ we take the natural filtration, generated by the restriction of the point process to $(0,t]$.

\medskip
By standard arguments, property (\ref{rank-M}) extends to any stopping time
$S$ with $0<S<\infty$:  
\[ \mbox{ 
Conditional on $\rank(M_\alpha(S)) = \bp$ and on $\FF(S)$, the ranked process 
} \]
\begin{equation}
\rank(M_\alpha(S + \cdot)) || \{M_\alpha(S)\}) \mbox{ is $\bp$-feasible.}
\label{rank-S}
 \end{equation}

In our initial definition of a $\bp$-feasible process we assumed the initial 
configuration $\bp$ was deterministic.  
Now define a $\oplus$-feasible process to be a mixture over $\bp$ of 
$\bp$-feasible processes; in other words, a process 
$(M_i(t), i \geq 1, t \ge 0)$ which, conditional on 
$(M_i(0), i \geq 1) = (p_i,i \ge 1)$, is a $\bp$-feasible process. 
So the ranked process $\rank(M_\alpha(S + \cdot)) || \{M_\alpha(S)\})$ in (\ref{rank-S}), considered unconditionally, 
is a $\oplus$-feasible process, and we describe the relationship 
(\ref{rank-S}) by saying this $\oplus$-feasible process is {\em embedded} into the 
$\bzero$-feasible process via the stopping time $S$.
Similarly, any stopping time within a  $\oplus$-feasible process 
specifies an embedded $\oplus$-feasible process.

 \begin{rmk}
\label{R1}
An essentially equivalent formalization would be to assign random 
$U[0,1]$ labels $U_\alpha$  to component martingales, so the state 
of the process at $t$ is described via the pairs $(U_\alpha, M_\alpha(t))$ 
for which $M_\alpha(t) > 0$, and this can in turn be described via the 
probability measure $\sum_\alpha M_\alpha(t) \delta_{U_\alpha}$ or its distribution function.  We will use this ``random labels" idea in an argument below.
\end{rmk}

\subsection{A general construction of $\bzero$-feasible 
processes}
\label{sec:construct}
 Given a  $\bzero$-feasible 
process and stopping times $S_k \downarrow 0$ a.s., 
the associated embedded $\oplus$-feasible processes are embedded within each other, and their initial values 
$(M^{(k)}_i, i \ge 1)$ satisfy 
$\max_i M^{(k)}_i \to 0$ a.s..
The following result formalizes the converse idea: one can  construct a  $\bzero$-feasible 
process from a sequence of $\bp$-feasible or more generally $\oplus$-feasible 
processes embedded into each other, via Kolmogorov consistency.
\begin{prop}
\label{Prop:consistent}
Suppose that $(\mu_k, k \geq 1)$ are probability measures on $\nabla$ 
and that for each $k$ there are families 
$(M^k_i(t), i \geq 1, 0 \le t < \infty)$ such that \\
(i) $(M^k_i(0), i \geq 1)$ has distribution $\mu_k$.\\
(ii) Conditional on $(M^k_i(0), i \geq 1) = \bp$, the process 
$\mathbf{M}^k =
(M^k_i(t), 0 \le t < \infty, i \geq 1)$ is $\bp$-feasible. \\
(iii) For $k \ge 2$ there is a stopping time $T_k$ for $\mathbf{M}^k$ such that $t_k := \Ex [T_k] < \infty$ and 
$\rank (M^k_i(T_k), i \ge 1)$ has distribution $\mu_{k-1}$. \\
(iv) $\sum_k t_k < \infty$.  \\
(v) $M_1^k(0) \to_p 0$ as $k \to \infty$.\\
Then there exists a $\bzero$-feasible process $\{M_\alpha(\cdot)\}$
 which is consistent 
with the families above, in the following sense.  
There exist stopping times $S_k$ such that for each $k \ge 1$
\[ \Ex[ S_k] = \sum_{j > k} t_j , \quad S_k - S_{k+1} \ed T_{k+1}
 \] 
and the  embedded process  
$\rank(M_\alpha(S_k + \cdot)) || \{M_\alpha(S_k)\}) $ is distributed as 
$M^k(\cdot)$.
\end{prop}
\begin{proof}
By conditions (i)-(iii), 
for each $k \ge 2$ we can represent the process $\mathbf{M}^{k-1}$ as the process 
$\mathbf{M}^k(T_k + \cdot)$; more precisely, we can couple the two processes 
such that
\begin{equation}
 M^{k-1}_i(t) = \rank(M^k_i(T_k + t) || (M^k_i(T_k), \ i \geq 1)) . 
\label{Mk-1}
\end{equation}
Then by the Kolmogorov consistency theorem we can assume this representation holds simultaneously 
for all $k$. 
We now attach labels $\alpha$ to the component martingales by the 
following inductive scheme.  
For $k = 1$, to each of the indices $i$ designating a component martingale $M^1_i(\cdot)$ we associate an independent 
Uniform$(0,1)$ label. 
For $k = 2$, a component martingale $M^2_i(\cdot)$ might be
zero or non-zero at $T_2$.  If non-zero then we 
copy the label already associated within $M^1(\cdot)$ 
via the coupling (\ref{Mk-1}).  If zero the we create a new independent 
Uniform$(0,1)$ label. 

Continue for each $k$ this scheme of copying or creating labels.
For each label $\alpha$, the sample path of that martingale component in
the process $M^{k+1}$ is obtained from the sample path in $M^k$ by 
inserting an extra initial segment.  By (iv) the path converges as $k \to \infty$ to 
a function $M_\alpha(t), 0 \le t < \infty$, and by (v) we must have 
$M_\alpha(0) = 0$.  
The remaining properties are straightforward.
\end{proof}

\subsection{All $\bp$-feasible processes embed}
\label{sec:all_embed}
Proposition \ref{Prop:consistent} enables construction of specific 
$\bzero$-feasible processes.  
The following result implies that any $\bp$-feasible process can be embedded into some $\bzero$-feasible process -- simply splice the $\bzero$-feasible process in the proposition to the given 
$\bp$-feasible process at time $S$.  
We already used this fact in the proofs of Corollary \ref{C1}(b) and Proposition \ref{LBDcrit}.
\begin{prop}
\label{P-embed}
Given any ranked $\bp$, there exists a $\bzero$-feasible 
process $\{M_\alpha(\cdot)\}$ and a stopping time $S$ such that 
$\rank( \ \{M_\alpha(S)\} \ ) = \bp$.
\end{prop}
For the proof it is convenient to  use Brownian-type process instead of 
Wright-Fisher. 
Write 
\[ Q(t) := \sum_i (M_i(t))^2 . \]
We will use constructions with the property
\[ \mbox{ At each time $0 \le t \le S$, at least one component} \]
\begin{equation}
 \mbox{ martingale $M_i(t)$ is evolving as Brownian motion } 
\label{Q-prop}
\end{equation}
for a specified stopping time $S$.  That is, our constructions can be  written as 
\[ dM_i(t) = \sigma_i(t) dW_i(t) \]
for (dependent) standard Brownian motions $W_i(t)$, and we require that 
some $\sigma_i(t)$ equals $1$.   
In general $Q(t) - \int_0^t\sum_i \sigma_i^2(s) \ ds$ is a martingale, so 
the advantage of  property  (\ref{Q-prop}) is that $Q(t) - t$ is a submartingale, implying
\begin{lemma}
\label{L:Qt}
Let $(M_i(t))$ be a $\bp$-feasible process satisfying (\ref{Q-prop}) for a stopping time $S$. 
Then 
$\Ex [S] \le \Ex[ Q(S)] - Q(0)$.
\end{lemma}

A simple construction satisfying  (\ref{Q-prop}) is the {\em Brownian reflection coupling} of two component martingales. 
That is, on $0 \le t \le S$ we freeze components other than $i,j$, and set 
\[ M_i(t) - M_i(0) = W_i(t), \quad M_j(t) - M_j(0) = - W_i(t) .\]

\begin{lemma}
\label{L:012}
Let $I_0$ be countable, and $I_1$ and $I_2$ be finite, index sets.
Let $(p_i, i \in I_0 \cup I_1)$ and $(q_i, i \in I_0 \cup I_2)$ be probability distributions which coincide on $I_0$ and satisfy 
$\max_{i \in I_1} p_i \le \min_{i \in I_2} q_i$. 
Then there exists a $\bp$-feasible process $\{M_\alpha(\cdot)\}$ satisfying (\ref{Q-prop}) 
such that for some stopping time $S$ we have 
$\rank( \ \{M_\alpha(S)\} \ ) = \rank(\bq)$.
\end{lemma}
\begin{proof} 
Freeze permanently the component martingales with $i \in I_0$.
Pick two arbitrary indices $i^\prime, i^*$ in $I_1$ and run the Brownian reflection coupling on these two components $M_{i^\prime}(t),M_{i^*}(t)$
until one component hits zero or $\min_{i \in I_2} q_i$.  
In the latter case, freeze that component permanently and delete its index from $I_1$ and delete $\arg \min_{i \in I_2} q_i$ from $I_2$.  In the former case, only delete the index from $I_1$. 
The total number (originally $|I_1| + |I_2|$) 
of unfrozen components is now decreased by at least $1$. 
Continue inductively, picking two  components from $I_1$ at each stage. 
Eventually all components are frozen and the ranked state is $\rank(\bq)$.
\end{proof}
\begin{proof}[Proof of Proposition \ref{P-embed}]
Define $\bp^0 = \bp$ and for $k \geq 1$ construct $\bp^k$ from $\bp$
by \\
(i) retaining entries $p_i$ with $p_i \le 2^{-k}$; \\
(ii)  replacing other $p_i$ by $2^{j(i)}$ copies of $2^{-j(i)}p_i$, 
where $j(i) \ge 1$ is the smallest integer such that $2^{-j(i)}p_i \le 2^{-k}$.

\noindent
Each pair $(\bp^k ,\bp^{k-1})$ satisfies the hypothesis of Lemma \ref{L:012}.  So for each $k$, writting $\mu_k = \delta_{\bp^k}$ and writing
$\bM^k$ and $T_k$ for the $\bp^k$-feasible process and the stopping time given by Lemma \ref{L:012}, 
we see that hypotheses (i)-(iii) of Proposition \ref{Prop:consistent} 
are satisfied.  
Moreover by Lemma \ref{L:Qt} we have 
$\Ex [T_k] \le q_{k-1} - q_k$ for 
$q_k:= \sum_i (p^k_i)^2$, implying that  hypotheses (i)-(iii) 
are also satisfied. 
The conclusion of Proposition \ref{Prop:consistent} 
now establishes  Proposition \ref{P-embed}. 
\end{proof}

\section{The $\bzero$-Wright-Fisher process}
\label{sec:0WF}
Write $\Delta$ for the (unranked) infinite simplex 
$\{(p_i, 1 \le i < \infty):\  p_i \ge 0, \sum_i p_i = 1\}$. 
As mentioned in section \ref{sec:multi}, 
for each $\bp \in \Delta$ there exists the $\bp$-Wright-Fisher process, 
a process with sample paths in 
$C([0,\infty), \Delta)$ and initial state $\bp$, which is
the infinite-dimensional diffusion with generator analogous to 
(\ref{WF-generator}) starting from state $\bp$, and that this is a 
$\bp$-feasible process.  
This has a straightforward construction: given $\bp \in \Delta$, 
set $\bp^n = (p_1,\ldots,p_{n-1}, \sum_{m \ge n} p_m)$, so the 
$\bp^n$-process exists as a finite-dimensional diffusion. 
But there is a natural coupling between the $\bp^{n-1}$- 
and the $\bp^n$-processes in which the first $n-2$ coordinate processes coincide, and appealing to Kolmogorov consistency for the infinite sequences of processes we immediately obtain the $\bp$-process.

Intuitively, we want to think of the $\bzero$-Wright-Fisher process as a suitable limit 
of the $(1/n,1/n,\ldots,1/n)$-Wright-Fisher processes as $n \to \infty$. 
But in fact the limit in distribution, in the compactified space
$\overline{\Delta} = \{(p_i, 1 \le i < \infty):\  p_i \ge 0, \sum_i p_i \le 1\}$,
is the process which is identically 
$(0,0,0,\ldots)$. 
The foundational 1981 paper of Ethier and Kurtz \cite{MR615945} shows that a non-trivial limit $\bX(t) = (X_i(t), i \ge 1)$ starting from $(0,0,0,\ldots)$ does exist if we work in the ranked infinite simplex $\nabla$; more precisely 
the limit process has sample paths in $C([0,\infty), \overline{\nabla})$ for 
 the compactifed ranked simplex $\overline{\nabla}$, but for $t>0$ takes values in $\nabla$.

That process is in some senses the process we want, but that formalization 
does not suffice for our purposes because it does not preserve the identity 
of components  as $t$ varies. 
That is,  we want the $\bzero$-feasible process $\{M_\alpha(t)\}$ 
whose components are martingales and for which 
\begin{equation}
\mbox{ $\bX(t) = \rank(\{M_\alpha(t)\})$ with a separate ranking for each $t$.} 
\label{bzz}
\end{equation}
The component processes $X_i(\cdot)$ are not martingales and  we cannot define 
quantities like $N_b$ and $D_{ab}$ in terms of $\bX$.  
Note that by Lemma \ref{L:naive} we cannot represent 
$\bX(t)$ as $\rank(\bM(t))$ for any process in $C([0,\infty), \overline{\Delta})$ 
with martingale components.

 Fortunately we can fit the $\bzero$-Wright-Fisher process into our abstract set-up by combining the existence of the process $\bX(t)$ with our
Proposition \ref{Prop:consistent}.  
Take times $s_k \downarrow 0$ and let $\mu_k$ be the distribution of 
$\bX(s_k)$.  
Then there is a $\oplus$-feasible Wright-Fisher process $\bM^k$ with initial distribution 
$\mu_k$, and existence of the ranked Wright-Fisher process $\bX$ implies that consistency condition (iii) of Proposition \ref{Prop:consistent} 
holds with $T_k = s_{k-1} - s_k$, and the conclusion of that proposition 
is that a $\bzero$-feasible process satisfying (\ref{bzz}) exists.

\subsection{Distributions associated with the $\bzero$-Wright-Fisher process}
\begin{prb}
\label{pr-1}
What are the distributions of $N_b$ and $D_{ab}$ for the $\bzero$-Wright-Fisher process?
\end{prb}
We remark that, if one only wanted to compute $\var(N_b)$, it would be sufficient to determine the limiting behavior of the quantity 
\begin{equation}\label{cov3}
\Pr(\sup_t M_1(t)\geq b,\,\sup_t M_2(t)\geq b|\,M_1(0)=x,\,M_2(0)=y) 
\end{equation}
in the limit $x,y\downarrow0$, where $M_1$, $M_2$ are the first two components of a $3$-allele Wright-Fisher diffusion. We also note that the quantity \eqref{cov3} coincides with the classical solution of the PDE $\frac{1}{2}x(1-x)f_{xx}+\frac{1}{2}y(1-y)f_{yy}-xy\,f_{xy}=0$ on $[0,b]\times[0,b]$ with the boundary conditions $f(x,0)=f(0,y)=0$, $f(x,b)=x/b$, $f(b,y)=y/b$, provided that such a solution exists. We were not able to solve the PDE explicitly, so that even the question of finding $\var(N_b)$ is an open problem. 

\section{Final remarks and open problems}\label{sec:OPext}
We have already stated open problem \ref{pr-1} and
Conjecture \ref{conj_1}.
The discussion at the start of section \ref{sec:small_spread} 
concerning constructions where $D_{ab}$ has small spread suggests 
the following closely analogous question concerning Brownian motions.
\begin{prb}
For each $1 \le i \le k$ let $(B_i(t), 0 \le t)$ be standard Brownian motion w.r.t. the same filtration, killed upon first hitting $-1$, and let $L_i$ be the 
total local time of $B_i(\cdot)$ at $0$.  
How small can the ratio 
$\var [\sum_{i=1}^k L_i] / \var [L_1]$ be?
\end{prb} 
We do not know any relevant work, though Jim Pitman (personal communication) observes that for $k=2$ one can indeed have negative correlation between $L_1$ and $L_2$.

\bibliographystyle{amsplain}
\bibliography{MGs}

\end{document}